\documentclass[12pt,reqno]{amsart}
\usepackage[title]{appendix}
\usepackage{bm}
\usepackage{cite}
\usepackage{calc}
\usepackage{cases}
\usepackage{comment}
\usepackage{enumerate}
\usepackage{geometry} 
\usepackage[linktocpage=true,colorlinks,citecolor=magenta,linkcolor=blue,urlcolor=magenta]{hyperref}
\usepackage{mathrsfs} 
\usepackage{url}
\usepackage{xcolor}

\theoremstyle{theorem}
\newtheorem{thm}{Theorem}[section]
\newtheorem{cor}[thm]{Corollary}

\newtheorem{lem}[thm]{Lemma}

\numberwithin{equation}{section}

\def\barlp{\overline{\Delta}}

\def\R{\mathbb{R}^{n+1}(c)}

\def\l{\langle}
\def\r{\rangle}

\begin{document}

\title[Triharmonic CMC hypersurfaces in space forms]{Triharmonic CMC hypersurfaces in space forms with 4 distinct principal curvatures}

\author[H. Chen ]{Hang Chen}
\address{School of Mathematics and Statistics, Northwestern Polytechnical University, Xi' an 710072,  P. R. China}
\email{\href{mailto:chenhang86@nwpu.edu.cn}{chenhang86@nwpu.edu.cn}}

\author[Z. Guan]{Zhida Guan}
\address{Department of Mathematical Sciences, Tsinghua University, Beijing 100084, P. R. China}
\email{\href{mailto:}{gzd15@mails.tsinghua.edu.cn}}

\begin{abstract}
A triharmonic map is a critical point of the tri-energy in the space of smooth maps between two Riemannian manifolds. In this paper, we prove that if $M^{n} (n\ge 4)$ is a CMC proper triharmonic hypersurface in a space form $\mathbb{R}^{n+1}(c)$ with four distinct principal curvatures and the multiplicity of the zero principal curvature is at most one, then $M$ has constant scalar curvature. In particular, we obtain any CMC proper triharmonic hypersurface in $\mathbb{R}^5(c)$ is minimal when $c\le 0$, which supports the generalized Chen's conjecture. We also give some characterizations of CMC proper triharmonic hypersurfaces in $\mathbb{S}^5$.
\end{abstract}



\maketitle

\section{Introduction}\label{sect:1}

A $k$-harmonic map $\phi$, introduced by Eells and Sampson \cite{ES1964}, is a critical point of the $k$-energy functional
\begin{equation*}
	E_k(\phi)=\frac{1}{2} \int_M \big|(d+d^{\ast})^k \phi\big|^2 v_g,
\end{equation*}
where $\phi:(M,g)\rightarrow(N,\bar{g})$ is a smooth map. It is equivalent to the $k$-tension field $\tau_k(\phi)\equiv 0$ from the Euler-Lagrange equations. This concept is a generalization of the classical ``harmonic map'' for $k=1$,
%
and the critical point of $E_k$ is usually called a biharmonic or triharmonic map for $k=2$ or $k=3$ respectively. G.~Jiang \cite{Jiang1986} computed the first and second variational formulae of $E_2$, and there have been a lot of results on biharmonic maps; the readers can refer to the very recent book by Ou and Chen \cite{OuChen}. For $k\ge 3$, the first and the second variational formulae of $E_k$ were obtained by S.~Wang \cite{Wang1989} and by Maeta \cite{Maeta2012a}  respectively, and $k$-harmonic maps have been widely studied as well, see \cite{ACL1983, BG1992,BL2002,Maeta2012,Maeta2012a,Maeta2015,NU2018,Ou2006} and the references therein. 

$M$ is called a $k$-harmonic submanifold of $N$ if $\phi: M\to N$ is an isometric and $k$-harmonic immersion.  A well-known fact says that  $M$ is harmonic if and only if $M$ is minimal since $\tau_1(\phi)=n\mathbf{H}$, where $n=\dim M$ and $\mathbf{H}$ denotes the mean curvature vector of $M$ in $N$. On the other hand, a harmonic map is always $k$-harmonic, but $k$-harmonic doesn't mean $l$-harmonic for $1\le l<k$ (cf. \cite{Wang1989, Maeta2012a}). 
Hence, a natural and important question raised in studying $k$-harmonic maps is that when a $k$-harmonic map $\phi$ (resp. a submanifold $M$) is  ``\emph{proper}'', namely, $\phi$ (resp. $M$) is not harmonic (resp. minimal). 
A famous conjecture says that any $k$-harmonic submanifold of the Euclidean space is minimal. This conjecture is usually called the generalized Chen's conjecture, since it was originally proposed by B.-Y. Chen \cite{Chen1991} for $k=2$ and extended to $k\ge 3$ by Maeta  \cite{Maeta2012}. This conjecture is not yet settled completely, but there are known partial results supporting it \cite{BFO2010,NU2011,NU2013,NUG2014,FHZ20}.
 
In this paper, we focus on the case $k=3$, which is a particularly active subject recently (cf. \cite{Maeta2015a,MNU2015,MOR19,MR18,MR18a,Wang1991}). For instance, Maeta-Nakauchi-Urakawa \cite{MNU2015} proved that a triharmonic isometric immersion into a Riemannian manifold of non-positive curvature must be minimal under certain suitable conditions. Maeta proved that any compact constant mean curvature (CMC in short) triharmonic hypersurface $M^n$ in $\mathbb{R}^{n+1}(c) (c\leq 0)$ is minimal (see \cite[Proposition 4.3]{Maeta2012}), and Montaldo-Oniciuc-Ratto removed the compactness assumption for $n=2$ very recently since there are at most two distinct principal curvatures for a surface $M^{2}$ (see \cite[Theorem 1.3]{MOR19}). Here $\mathbb{R}^{n+1}(c)$ represents the $(n+1)$-dimensional Euclidean space $\mathbb{R}^{n+1}$, the hyperbolic space $\mathbb{H}^{n+1}$ and the unit sphere $\mathbb{S}^{n+1}$ for $c = 0, -1$ and $1$ respectively.  

In the previous paper, we considered the case of at most three distinct principal curvatures and proved 
\begin{thm}[{\cite[Theorem 1.1]{CG21}}]\label{thmA}
	Let $M^{n} (n\ge 3)$ be a CMC proper triharmonic hypersurface with at most three distinct principal curvatures in $\R$. 
	Then the scalar curvature of $M$ is constant.
\end{thm}

In this paper, we consider case of 4 distinct principal curvatures and prove the following theorem. 

\begin{thm}\label{thm1.1}
	Let $M^{n} (n\ge 4)$ be a CMC proper triharmonic hypersurface with 4 distinct principal curvatures in $\mathbb{R}^{n+1}(c)$. If the multiplicity of the zero principle curvature is at most one, then the scalar curvature of $M$ is constant. 
\end{thm}

The Gauss equation \eqref{gauss} implies that, the scalar curvature being constant is equivalent to the squared norm of the second fundamental form being constant for a CMC hypersurface in the space form.  Therefore, as a direct consequence of Theorem \ref{thm1.1}, we obtain
\begin{cor}\label{cor1}
	Let $M^{n} (n\ge 4)$ be a CMC triharmonic hypersurface with four distinct principal curvatures in $\R (c\le 0)$. If the multiplicity of the zero principle curvature is at most one, then $M$ must be minimal.
	
	In particular, any CMC triharmonic hypersurface in $\mathbb{R}^{5}(c) (c\le 0)$ must be minimal.
\end{cor}

When $c=1$, some examples of non-minimal $k$-harmonic hypersurfaces in a sphere are given and some classifications and characterizations are obtained, see \cite{WW12, MOR19, CG21} for instance. Based on these results and Theorem \ref{thm1.1}, we have

\begin{cor}\label{cor1.3}
	Let $M^{n} (n\ge 4)$ be a proper CMC  triharmonic hypersurface with four distinct principal curvatures in $\mathbb{S}^{n+1}$. If the multiplicity of the zero principle curvature is at most one, then either 
	
	(1) $H^{2}=2$ and $M$ is locally $\mathbb{S}^{n}(1/\sqrt{3})$; or
	
	(2) $H^2\in (0, t_0]$, and $H^{2}=t_0$ if and only $M$ is locally $\mathbb{S}^{n-1}(a)\times\mathbb{S}^{1}(\sqrt{1-a^{2}})$. Here $t_0$ is the unique real root belonging to $(0, 2)$ of the polynomial 
	\begin{equation*}
		f_{n,3}(t)=n^4t^3-2n^2(n^2-5n+5)t^2-(n-1)(2n-5)(3n-5)t-(n-1)(n-2)^2,
	\end{equation*}
	and the radius $a$ is given by
	\begin{equation}\label{eq-rad}
		a^{2}=\frac{2(n-1)^{2}}{n^{2}H^{2}+2n(n-1)+nH\sqrt{n^{2}H^{2}+4(n-1)}}.
	\end{equation}

In particular, the conclusions hold for a proper CMC triharmonic hypersurface in $\mathbb{S}^{5}$ without any extra assumptions.
\end{cor}

The paper is organized as follows.
In Sect.~\ref{sect:2}, we introduce some notations and  recall some fundamental concepts and formulae for triharmonic hypersurfaces in space forms. In Sect.~\ref{sect:3}, we show some lemmas that will be used in the proofs of main theorems. In Sect.~\ref{sect:4}, we use proof by contradiction to prove Theorem \ref{thm1.1}; we also prove the corollaries. 

\textbf{Acknowledgment}:
The first author was partially supported by Natural Science Foundation of Shannxi Province Grant No.~2020JQ-101 and the Fundamental Research Funds for the Central Universities Grant No.~310201911cx013.
The second author was partially supported by NSFC Grant No.~11831005 and No.~11671224.
The authors would like to thank Professor Haizhong Li for bringing the question to our attention and his very helpful suggestions and comments.

\section{Preliminaries and notations}\label{sect:2}
Throughout this paper, we will use  $i,j,k, \ldots$ for indices running over $\{1,\ldots , n\}$ unless otherwise declaration.

\subsection{Fundamental formulae of hypersurfaces in 
$\R$}
Let $M$ be an $n$-dimensional hypersurface in the space form $\R$. Let $\nabla$ and $\bar{\nabla}$ denote the Levi-Civita connections on $M$ and $\R$ respectively. The (1,3)-type and (0,4)-type Riemannian curvature tensors of $M$ are respectively given by 
\begin{align*}
	R(X,Y)Z&=(\nabla_X\nabla_Y-\nabla_Y\nabla_X-\nabla_{[X,Y]})Z,\\
	R(X,Y,Z,W)&=\l R(X,Y)W, Z\r;
\end{align*}
 the Gauss and Weingarten formulae are respectively given by 
\begin{align*}
	\bar{\nabla}_XY=\nabla_XY+h(X,Y)\xi,\quad
	\bar{\nabla}_{X}\xi=-A(X).
\end{align*}
Here $X, Y, Z, W$ are tangent vector fields on $M$,  $\xi$ is the unit normal vector field on $M$, $h$ is the second fundamental form of $M$, and $A$ is the shape operator. 

We choose an orthonormal frame $\{e_i\}_{i=1}^n$ of $M$ and suppose $\nabla_{e_i}e_j=\sum_k\Gamma_{ij}^{k}e_k$.
By denoting 
$R_{ijkl}=R(e_i,e_j,e_k,e_l), h_{ij}=h(e_i,e_j), h_{ijk}=e_kh_{ij}-h(\nabla_{e_k}e_i,e_j)-h(e_i,\nabla_{e_k}e_j)$,
 we obtain
\begin{align}
R_{ijkl}&=(e_i\Gamma_{jl}^{k}-e_j\Gamma_{il}^{k})+\sum_{m}(\Gamma_{jl}^{m}\Gamma_{im}^{k}-\Gamma_{il}^{m}\Gamma_{jm}^{k}-(\Gamma_{ij}^{m}-\Gamma_{ji}^{m})\Gamma_{ml}^{k}),\label{eq-curva}\\
h_{ijk}&=e_kh_{ij}-\sum_{l}(\Gamma_{ki}^{l}h_{lj}+\Gamma_{kj}^{l}h_{il}).\label{eq-hijk}
\end{align}
Hence, the Gauss-Codazzi equations are respectively given by
\begin{align}\label{gauss}
R_{ijkl}&=(\delta_{ik}\delta_{jl}-\delta_{il}\delta_{jk})c+( h_{ik}h_{jl}-h_{il}h_{jk}),\\
\label{codazzi}
h_{ijk}&=h_{ikj}.
\end{align}

\subsection{Triharmonic hypersurfaces in $\R$}
Let $\phi: (M^{n},g)\to (N,\bar{g})$.
The first variational formula of $E_k$ (cf. \cite{Wang1989}) for $k=3$ gives
\begin{equation*}
\tau_3(\phi)=\barlp^{2}\tau(\phi)-\sum_{i}\bar{R}\big(\barlp\tau(\phi),d\phi(e_i)\big)d\phi(e_i)-\sum_{i}\bar{R}\big(\nabla^{\phi}_{e_i}\tau(\phi),\tau(\phi)\big)d\phi(e_i),
\end{equation*}
where $\nabla^{\phi}$ is the induced connection on the bundle $\phi^{-1}TN$ and  $\barlp=-\sum_{i}(\nabla^{\phi}_{e_i}\nabla^{\phi}_{e_i}-\nabla^{\phi}_{\nabla_{e_i}e_i})$ is the \emph{rough Laplacian} on the section of $\phi^{-1}TN$. 

By a direct computation (cf. \cite[Eq. (2.8)]{MOR19}), a CMC hypersurface of $\R$ is triharmonic if and only if 
\begin{equation}\label{tri-cmc}
\begin{cases}
H(\Delta S+S^2-cnS-cn^{2}H^{2})=0, \\
HA(\nabla S)=0.
\end{cases}
\end{equation}
Here $H=\frac{1}{n}
(\sum\limits_i h_{ii})$ and  $S=\sum\limits_{i,j}(h_{ij})^2$ denote the mean curvature  function of $M$ and the
squared norm of the second fundamental form respectively. We reminder the readers that the Laplacian 
$\Delta=-\sum_{i}(\nabla_{e_i}\nabla_{e_i}-\nabla_{\nabla_{e_i}e_i})$ here has a sign opposite to the usual definition.

Clearly \eqref{tri-cmc} always holds for $H=0$, which is the trivial case. From now on, we assume that $M$ is not minimal, then \eqref{tri-cmc} becomes
\begin{subnumcases}{}
(\Delta S+S^2-cnS-cn^{2}H^{2})=0 \label{tri-1'}\\
A(\nabla S)=0.\label{tri-2'}
\end{subnumcases}

\section{Some Lemmas}\label{sect:3}

In this section, we consider the general case of $d (\ge 4)$ distinct principal curvatures and prove some lemmas.

Obviously  $\mathcal{N}:=\{p\in M: \nabla S(p)\neq 0 \}$ is an open set. From now on, we assume that  $\mathcal{N}\neq \emptyset$ and work in $\mathcal{N}$.  It follows from \eqref{tri-2'} that $\nabla S$ is a principal direction with the corresponding principal curvature $0$. Without loss of generality, we can choose an orthonormal frame $\{e_i\}$ such that $e_1$ is parallel to $\nabla S$ and the shape operator $A$ is diagonalized with respect to $\{e_i\}$, i.e.,  $h_{ij}=\lambda_i\delta_{ij}$, where $\lambda_i$ is the principal curvature and $\lambda_1=0$.

Assume that $M$ has $d$ distinct principal curvatures $\mu_1=0, \mu_2,\cdots,\mu_d$ with $d\ge 4$. Without loss of generality, we suppose 
\begin{equation}
	\lambda_i=\mu_\alpha \mbox{ when $i\in I_\alpha$,}
\end{equation}
where 
\begin{equation}
	I_\alpha=\Big\{\sum_{0\le \beta\le \alpha-1}n_\beta+1,\cdots, \sum_{0\le \beta\le \alpha}n_\beta\Big\}
\end{equation}
with $n_0=0$ and $n_\alpha\in \mathbb{Z}_{+} $  satisfying $\sum_{1\le \alpha\le d}n_\alpha=n$, namely, $n_\alpha$ is the multiplicity of $\mu_\alpha$.
For convenience, we will use the range of indices $1\le \alpha,\beta, \gamma,\cdots\le d$ except special declaration. Then we have the following lemma.

\begin{lem}\label{lem-Gijk}
 The coefficient $\Gamma_{ij}^{k}$ satisfies:
	\begin{enumerate} 
		\item \label{G1}$\Gamma_{ij}^{k}=-\Gamma_{ik}^{j}$.
		\item \label{G2} $\Gamma_{ii}^{k}=\frac{e_k\lambda_i}{\lambda_i-\lambda_k}$ for $i\in I_\alpha$ and $k\notin I_\alpha$.
		\item $\Gamma_{ij}^{k}=\Gamma_{ji}^{k}$ if the indices satisfy one of the following conditions:
		\begin{enumerate}[(3a)] \label{G3}
			\item \label{G3-1} $i,j\in I_\alpha$ but $k\notin I_\alpha$;
			\item \label{G3-2}$i,j\ge 2$ and $k=1$.
		\end{enumerate}
		\item $\Gamma_{ij}^{k}=0$ if the indices satisfy one of the following conditions:
		\begin{enumerate}[(4a)] \label{G4}
			\item \label{G4-1}$j=k$;
			\item \label{G4-3} $i=j\in I_1$ and $k\notin I_1$;
			\item \label{G4-4}$i,k\in I_\alpha, i\neq k$ and $j\notin I_\alpha$;
			\item \label{G4-5} $i,j\ge 2, i\in I_\alpha, j\in I_\beta $ with $\alpha\neq \beta$ and $k=1$.
		\end{enumerate}
		\item \label{G2'}
	$\Gamma_{ji}^{k}=\frac{\lambda_j-\lambda_k}{\lambda_i-\lambda_k}\Gamma_{ij}^{k}, \Gamma_{ki}^{j}=\frac{\lambda_k-\lambda_j}{\lambda_i-\lambda_j}\Gamma_{ik}^{j}$ for $\lambda_i,\lambda_j$ and $\lambda_k$ are mutually different.
	\item \label{G2''}
	$\Gamma_{ij}^{k}\Gamma_{ji}^{k}+\Gamma_{ik}^{j}\Gamma_{ki}^{j}+\Gamma_{jk}^{i}\Gamma_{kj}^{i}=0$  for $\lambda_i,\lambda_j$ and $\lambda_k$ are mutually different.
	\end{enumerate}
\end{lem}
\begin{proof}
	Items \eqref{G1}--\eqref{G4} were proved in \cite[Lemma 3.1]{CG21} for the case of 3 distinct principal curvatures, and the proof is also valid for $d\ge 4$. Here we restate it for the readers' convenience, and also give the proof of Items \eqref{G2'} and \eqref{G2''}.
	
	Item \eqref{G1} is directly from $e_i\l e_j,e_k \r=0$, which also implies \eqref{G4-1}.
	
	By  \eqref{eq-hijk}, \eqref{codazzi} and \eqref{G1}, we have 
	\begin{equation}\label{eq-1}
		e_k\lambda_i=h_{iik}=h_{iki}=\Gamma_{ii}^{k}(\lambda_i-\lambda_k) \mbox{ for $i\neq k$,}
	\end{equation}
	\begin{equation}\label{eq-2}
		\Gamma_{ij}^{k}(\lambda_j-\lambda_k)=h_{kji}=h_{kij}=\Gamma_{ji}^{k}(\lambda_i-\lambda_k) \mbox{ for distinct $i,j,k$.}
	\end{equation}
	
	Items \eqref{G2} and \eqref{G2'} follow from \eqref{eq-1} and \eqref{eq-2} respectively, and 
	Item \eqref{G2''} follows from Items \eqref{G1} and \eqref{G2'}.

	\eqref{G3-1} is directly from \eqref{eq-2} since $\lambda_j-\lambda_k=\lambda_i-\lambda_k\neq 0$.
		
		By the assumptions and the choice of $\{e_i\}$, we have $e_1S\neq 0, e_iS=0 (2\le i \le n)$. Then for $2\le i,j\le n, 0=(e_ie_j-e_ie_i)S=[e_i,e_j]S=(\nabla_{e_i}e_j-\nabla_{e_j}e_i)S=(\Gamma_{ij}^{1}-\Gamma_{ji}^{1})e_1S$, so $\Gamma_{ij}^{1}=\Gamma_{ji}^{1}.$ That is \eqref{G3-2}.
	
	\eqref{G4-3} is directly from \eqref{eq-1} since $e_k\lambda_i=0$ for $i\in I_1$.

	Again from \eqref{eq-2} we have $\Gamma_{ij}^{k}(\lambda_j-\lambda_k)=0$ for $i,k\in I_p, j\notin I_p, i\neq k$. Since $\lambda_j-\lambda_k\neq 0$, we obtain \eqref{G4-4}. 
	
	If $i,j\ge 2$, then from \eqref{eq-2} and \eqref{G3-2} we can derive $\Gamma_{ij}^{1}(\mu_\beta-\mu_\alpha)=0$. Since $\alpha\neq \beta$, we have $\mu_\beta-\mu_\alpha\neq 0$ and then obtain \eqref{G4-5}. 	
\end{proof}

\begin{lem}\label{lem2}
	For $2\le \alpha\le d$, denote $P_\alpha=\frac{e_1\mu_\alpha}{\mu_\alpha}$. If the multiplicity of the zero principal curvature is one, then we have 
	\begin{equation}\label{eq-lem2}
	e_1P_\alpha-P_\alpha^{2}=c.
	\end{equation}
\end{lem}
\begin{proof}
 By the assumption, we have $I_1=\{1\}$.
	For $j\in I_{\alpha}$, it follows from   \eqref{eq-curva} and Lemma
	  \ref{lem-Gijk} that 
	\begin{align*}  
	 R_{1j1j}&=(e_1\Gamma_{jj}^{1}-e_j\Gamma_{1j}^{1})+\sum_{m}(\Gamma_{jj}^{m}\Gamma_{1m}^{1}-\Gamma_{1j}^{m}\Gamma_{jm}^{1}-\Gamma_{1j}^{m}\Gamma_{mj}^{1}+\Gamma_{j1}^{m}\Gamma_{mj}^{1}) \\
	 &=(e_1\Gamma_{jj}^{1}+e_j\Gamma_{11}^{j})-\sum_{m\ge 2}(\Gamma_{jj}^{m}\Gamma_{11}^{m}+\Gamma_{1j}^{m}\Gamma_{jm}^{1}-\Gamma_{1j}^{m}\Gamma_{m1}^{j}-\Gamma_{j1}^{m}\Gamma_{mj}^{1}) \\
	 &=e_1\Gamma_{jj}^{1}+\sum_{m\in  I_\alpha}(\Gamma_{1j}^{m}\Gamma_{j1}^{m}+\Gamma_{1j}^{m}\Gamma_{m1}^{j}+\Gamma_{j1}^{m}\Gamma_{mj}^{1})-\sum_{m\notin I_1\cup I_\alpha}(\Gamma_{1j}^{m}\Gamma_{jm}^{1}+\Gamma_{1j}^{m}\Gamma_{mj}^{1}-\Gamma_{j1}^{m}\Gamma_{mj}^{1})\\
	 &=e_1\Gamma_{jj}^{1}+\Gamma_{j1}^{j}\Gamma_{jj}^{1}=e_1\Gamma_{jj}^{1}-(\Gamma_{jj}^{1})^2=e_1P_\alpha-P_\alpha^2.
	\end{align*}	
On the other hand, 	
	 we have  $R_{1j1j}=c$ from the Gauss equation \eqref{gauss}, so we complete the proof.	
\end{proof}

\begin{lem}\label{lem3}
	Notations and assumptions are as in Lemma \ref{lem2}. If $M$ has constant mean curvature, then for any $q\in \mathbb{Z}_+$ we have 
	\begin{equation}\label{eq3.6}
	f(q):=\sum_{2\le \alpha\le d}n_{\alpha}\mu_\alpha P_\alpha^{q}=\begin{cases}
			0, & \mbox{when $q$ is odd,},\\
		\dfrac{(q-1)!! }{q!!}(-c)^{q/2}nH, &  \mbox{when $q$ is even.}
		\end{cases}
	\end{equation}
\end{lem}
\begin{proof}
	Since $H$ is constant, differentiating $nH=\sum_{2\le \alpha\le d}n_{\alpha}\mu_\alpha$  by $e_1$, we obtain $0=\sum_{2\le \alpha\le d}n_{\alpha}e_1\mu_\alpha$, that is, 
	\begin{equation}\label{P1}
		\sum_{2\le \alpha\le d}n_{\alpha}\mu_\alpha P_\alpha=0.
	\end{equation}
	Differentiating \eqref{P1} by $e_1$, from \eqref{eq-lem2} we have
	\begin{align*}
		0&=\sum_{2\le \alpha\le d}n_{\alpha}\big((e_1\mu_\alpha) P_\alpha+\mu_\alpha e_1(P_\alpha)\big)\\
		&=\sum_{2\le \alpha\le d}n_{\alpha}\big(\mu_\alpha P_\alpha^2+\mu_\alpha (P_\alpha^2+c)\big)
	\\
		&=2\sum_{2\le \alpha\le d}n_{\alpha}\mu_\alpha P_\alpha^2+cnH,
	\end{align*}
	which is equivalent to 
	\begin{align}\label{P2}
		\sum_{2\le \alpha\le d}n_{\alpha}\mu_\alpha P_\alpha^2=-\frac{1}{2}cnH.
	\end{align}
We have obtained that \eqref{eq3.6} holds for $q=1, 2$. so we can prove that it holds for general $q$ by induction. 

Whenever $q\ge 2$ is even or odd, we differentiate   \eqref{eq3.6} by $e_1$ and then obtain that 
\begin{align*}
	0&= \sum_{2\le \alpha\le d}n_{\alpha}\big((e_1\mu_\alpha) P_\alpha^q+q\mu_\alpha P_\alpha ^{q-1}e_1(P_\alpha)\big)\\
	&= \sum_{2\le \alpha\le d}n_{\alpha}\big(\mu_\alpha P_\alpha^{q+1}+q\mu_\alpha P_\alpha ^{q-1}(P_\alpha^{2}+c)\big) \\
	&=(q+1)f(q+1)+qcf(q-1).
\end{align*}

If $q$ is even, then both $q-1$ and $q+1$ are odd. So $f(q+1)=0$ is from $f(q-1)=0$.

If $q$ is odd, then both $q-1$ and $q+1$ are even. So  
\begin{align*}
	f(q+1)=&-\frac{qc}{q+1}f(q-1)
	=-\frac{qc}{q+1}\times\dfrac{(q-2)!! }{(q-1)!!}(-c)^{(q-1)/2}nH\\
	=&\dfrac{q!! }{(q+1)!!}(-c)^{(q+1)/2}nH.
\end{align*}
\end{proof}

\section{Proofs of Theorem \ref{thm1.1} and its corollaries}\label{sect:4}
In this section, we will prove Theorem \ref{thm1.1} and its corollaries.
\begin{proof}[Proof of Theorem \ref{thm1.1}]
We derive a contradiction from the assumption that $\mathcal{N}=\{p\in M: \nabla S(p)\neq 0 \}\neq \emptyset$. Under this condition, we can use Lemmas in Sect.~\ref{sect:3}.
	
Since $d=4$ and there are 3 distinct non-zero principal curvatures, we replace $\mu_2, \mu_3, \mu_4$ by $\lambda, \mu, \nu$ respectively, and replace $P_2, P_3, P_4$ by $P, Q, R$ respectively.

By taking $q=1, 3, 5$ respectively in Lemma \ref{lem3}, we obtain a linear system of the principal curvatures $\lambda,\mu,\nu$
\begin{subnumcases}{\label{lin-sys}}
r P\lambda+s Q\mu+t R\nu=0, \label{lin-sys1}\\
r P^3\lambda+s Q^3\mu+t R^3\nu=0,\label{lin-sys2} \\
r P^5\lambda+s Q^5\mu+t R^5\nu=0, \label{lin-sys3}
\end{subnumcases}
which has a non-zero solution. Hence, on $ \mathcal{N}$ we must have 
\begin{equation*}
\det
\begin{bmatrix}
rP & sQ & tR \\
rP^3 & sQ^3 & tR^3 \\
rP^5 & sQ^5 & tR^5
\end{bmatrix}
=rstPQR(P^2-Q^2)(Q^2-R^2)(R^2-P^2)=0.
\end{equation*}

Besides, by taking $q=2, 4$ respectively in Lemma \ref{lem3}, we obtain
\begin{subnumcases}{\label{lin-sys-even}}
	r P^2\lambda+s Q^2\mu+t R^2\nu=-\dfrac{1}{2}ncH,\label{lin-sys-even1} \\
	r P^4\lambda+s Q^4\mu+t R^4\nu=\dfrac{3}{8}nc^2H.\label{lin-sys-even2}
\end{subnumcases}

Now we check all possible cases.

\textbf{Case (1):} $PQR\neq 0$ at some $p\in \mathcal{N}$. 
Then $(P^2-Q^2)(Q^2-R^2)(R^2-P^2)=0$ at $p$. 
Without loss of generality, we assume $P^2-Q^2=0$,  i.e., $P=\pm Q$. Now \eqref{lin-sys1} and \eqref{lin-sys2} become
\begin{equation*}
	\begin{cases}
		(r\lambda\pm s\mu)P +t\nu R=0, \\
		(r\lambda\pm s\mu)P^3 +t\nu R^3=0. 
	\end{cases}
\end{equation*}
Since $t\nu$ is not zero, we have
\begin{equation*}
	\det\begin{bmatrix}
		P & R \\
		P^3 & R^3 
	\end{bmatrix}=PR(R^2-P^2)=0,
\end{equation*}
so $R^2-P^2=0$. Hence, $P^2=Q^2=R^2$ at $p$. Now \eqref{lin-sys-even} becomes
	\begin{subnumcases}{}
	-\dfrac{1}{2}ncH=(r\lambda+s\mu+t\nu)P^2=nP^2H,
	\label{4.3}\\
	\label{PP4}
	\dfrac{3}{8}nc^2H=(r\lambda+s\mu+t\nu)P^4=nP^4H.\label{4.4}
	\end{subnumcases}
Since $H\neq 0$, from \eqref{4.3} and \eqref{4.4} we obtain $c=-2P^2$ and $3c^2=8P^4$ respectively, which imply $c=P=0$, a contradiction.

\textbf{Case (2):} $PQR=0$ on  $\mathcal{N}$, since we've showed Case (1) is impossible. This means there is at lease one quantity among $P, Q, R$ equaling to $0$. If we know two quantities among $P, Q, R$ being 0, then the third one must be 0 from \eqref{lin-sys1}.

\textbf{Case (2.1):} $P=Q=R=0$ at some $p\in\mathcal{N}$.  We have 
\begin{align*}
	\frac{1}{2}e_1S&=\frac{1}{2}e_1(r\lambda^2+s\mu^2+t\nu^2) \\
	&=r\lambda e_1\lambda+s\mu e_1\mu+t\nu e_1\nu \\
	&=r\lambda^2 P+s\mu^2 Q+t\nu^2 R=0 \mbox{ at $p$,}
\end{align*}
which contradicts $\nabla S|_p\neq 0$ (recall that $e_mS$  always equals to 0 for $m\ge 2$).

\textbf{Case (2.2):} We've ruled out Case (2.1), so  for any given $p\in\mathcal{N}$, without loss of generality, we can assume $P=0, Q\neq 0, R\neq 0$ at $p$, and then this also holds on a neighborhood $U$ containing $p$. The linear system \eqref{lin-sys} implies
\begin{subnumcases}{}
 sQ\mu+ tR\nu=0, \label{4.4a}\\
 sQ^3\mu+ tR^3\nu=0.\label{4.4b}
\end{subnumcases}
Since $\mu,\nu$ are not zero, we have
\begin{equation*}
\det\begin{bmatrix}
sQ & tR \\
sQ^3 & tR^3 
\end{bmatrix}=stQR(R^2-Q^2)=0.
\end{equation*}
Since $QR\neq 0$, we have $Q^2-R^2 = 0$  on $U$. We can assume $U$ is connected, then there are two cases.

\textbf{Case (1.2.1):} $Q= -R$ on $U$. 
It follows from \eqref{4.4a} that $s\mu=t\nu$ . Further we have $se_1\mu= te_1\nu$, i.e., $s\mu Q=t\nu R$. 
Since $s\mu=t\nu\ne 0$, we derive that $Q=R$ and then $Q=R=0$, a contradiction. 

\textbf{Case (1.2.2):} $Q= R$ on $U$. From \eqref{4.4a} we have $s\mu+t\nu=0$ and then $\lambda=nH/r$ is constant since $r\lambda+s\mu+t\nu=nH$. Hence, $e_m\lambda=0$ for $1\le m \le n$.

On the other hand,  $0=\dfrac{1}{2}e_mS=\dfrac{1}{2}e_m(r\lambda^2+s\mu^2+t\nu^2)=s\mu e_m\mu+t\nu e_m\nu$ for $2\le m\le n$, so we have
\begin{equation*}
\begin{cases}
se_m\mu +te_m\nu =0, \\
s\mu e_m\mu +t\nu e_m\nu =0,
\end{cases}
\end{equation*}
which follows $e_m \mu=e_m \nu=0$ for $2\le m\le n$,

Let $j\in I_2, k\in I_3, l\in I_4$. We derive
\begin{equation}\label{eq4.1}
	\Gamma_{kk}^{j}=\Gamma_{jj}^{k}=\Gamma_{jj}^{l}=\Gamma_{kk}^{l}=\Gamma_{ll}^{k}=\Gamma_{ll}^{j}=0
\end{equation}
 from Item \eqref{G2} of Lemma \ref{lem-Gijk}.

From \eqref{lin-sys-even1} we have
$-\dfrac{1}{2}ncH=s\mu Q^2+t\nu R^2=0$ since $s\mu +t\nu =0$ and $Q^{2}=R^{2}$. This can only happen when $c=0$ since $H\neq 0$. When $c=0$, by Eq.~\eqref{eq-curva} and  Lemma \ref{lem-Gijk}, the Gauss equation \eqref{gauss} becomes
	\begin{align*}
	 \lambda\mu=R_{jkjk}&=(e_j\Gamma_{kk}^{j}-e_k\Gamma_{jk}^{j})+\sum_{1\le m\le n}(\Gamma_{kk}^{m}\Gamma_{jm}^{j}-\Gamma_{jk}^{m}\Gamma_{km}^{j}-\Gamma_{jk}^{m}\Gamma_{mk}^{j}+\Gamma_{kj}^{m}\Gamma_{mk}^{j}) \\
	 &=(e_j\Gamma_{kk}^{j}+e_k\Gamma_{jj}^{k})+\sum_{1\le m\le n}(-\Gamma_{jj}^{m}\Gamma_{kk}^{m}+\Gamma_{jk}^{m}\Gamma_{kj}^{m}-\Gamma_{jm}^{k}\Gamma_{mj}^{k}-\Gamma_{km}^{j}\Gamma_{mk}^{j})\\
	 &\overset{\text{(1)}}{=}\sum_{m\in I_4}(\Gamma_{jk}^{m}\Gamma_{kj}^{m}-\Gamma_{jm}^{k}\Gamma_{mj}^{k}-\Gamma_{km}^{j}\Gamma_{mk}^{j})\\
	 &\overset{\text{(2)}}{=}2\sum_{l}\Gamma_{jk}^{l}\Gamma_{kj}^{l} \overset{\text{(3)}}{=}2\sum_{l}\dfrac{\mu-\nu}{\lambda-\nu}(\Gamma_{jk}^{l})^2.
	\end{align*}
Here we used \eqref{eq4.1} and Item \eqref{G4} of Lemma \ref{lem-Gijk} in $\overset{\text{(1)}}{=}$, Item \eqref{G2''} of Lemma \ref{lem-Gijk} in $\overset{\text{(2)}}{=}$, and Item \eqref{G2'} of Lemma \ref{lem-Gijk} in $\overset{\text{(3)}}{=}$. Therefore, we obtain 
	\begin{equation}\label{lambdamu}
	s\lambda\mu=2\sum_{k, l}\dfrac{\mu-\nu}{\lambda-\nu}(\Gamma_{jk}^{l})^2.
	\end{equation}
Similarly, we have 
	\begin{equation}\label{lambdanu}
	t\lambda\nu=2\sum_{k, l}\dfrac{\nu-\mu}{\lambda-\mu}(\Gamma_{jl}^{k})^2.
	\end{equation}
Summing up \eqref{lambdamu} and \eqref{lambdanu} we conclude that
	\begin{align*}
	0=\lambda(s\mu+t\nu)=2\sum_{k, l}(\Gamma_{jk}^{l})^2(\mu-\nu)(\frac{1}{\lambda-\nu}-\frac{1}{\lambda-\mu}),
	\end{align*}
and then $\Gamma_{jk}^{l}=0$ for any $j\in I_2, k\in I_3, l\in I_4$, which follows from \eqref{lambdamu} that $\lambda\mu=0$, a contradiction.

In conclusion, neither  Case (1)  nor  Case (2)  can happen. Therefore, $\mathcal{N}$ must be empty and we complete the proof of Theorem \ref{thm1.1}.
\end{proof}

\begin{proof}[Proof of Corollary \ref{cor1}]
	If $H\neq 0$,  by applying Theorem \ref{thm1.1} we obtain $S^2\le 0$ from \eqref{tri-1'}, which gives a contradiction. 
	
	When the dimension $n=4$, the condition of 4 distinct principal curvatures implies the multiplicity of the zero principle curvature is at most one. Otherwise, there are at most 3 distinct principal curvatures, then the conclusion is from \cite[Corollary 1.2]{CG21} (a direct corollary of Theorem \ref{thmA}).
\end{proof}

\begin{proof}[Proof of Corollary \ref{cor1.3}]
	The conclusions follow directly from Theorem \ref{thm1.1},  \cite[Theorem 1.9]{MOR19} and \cite[Theorem 1.5]{CG21}. 
	
	Combining with  \cite[Corollary 1.7]{CG21}, the argument in the case $n=4$ is analogous as in the proof of Corollary \ref{cor1}. We omit the details.
\end{proof}

\end{document}